\documentclass[reqno]{amsart}
\usepackage{amssymb,amsmath,amsfonts,amsthm}
\usepackage{graphics,ifthen,mathrsfs,verbatim,epsfig,color}
\usepackage[mathcal]{euscript}
\usepackage{url}
\usepackage{hyperref}

\usepackage{tikz}

\theoremstyle{plain}
\newtheorem{theorem}  {Theorem}  [section]
\newtheorem{lemma}  [theorem]   {Lemma}
\newtheorem{corollary}[theorem] {Corollary}

\newtheorem{fact} [theorem] {Fact}
\newtheorem{proposition} [theorem] {Proposition}
\newtheorem{claim} [theorem] {Claim}
\newtheorem*{claim*} {Claim}
\newtheorem{conjecture}[theorem] {Conjecture}

\newtheorem{algorithm}{Algorithm}

\theoremstyle{definition}

\newtheorem{definition}[theorem] {Definition}
\newtheorem{example}[theorem] {Example}

\newcommand{\claimproof}{\renewcommand{\qedsymbol}{$\diamond$}}

\newcommand{\comp}[1]{{#1^{\rm c}}}
\newcommand{\redcomp}[1]{{\rm red}({#1^{\rm c}})}

\newcommand{\n}{\noindent}
\renewcommand{\setminus}{-}

\newcommand{\iso}[1][]{\buildrel {#1} \over \cong}

\newcommand{\lcov}{\prec}

  \def\R{{R}}

\def\c #1{{\mathcal #1}}
     \def\cC{\c{C}} \def\cD{\c{D}} \def\cE{\c{E}}
   \def\cG{\c{G}}     
    
  \def\cP{\c{P}}    
   \def\cV{\c{V}}

\let\polishlcross=\l
\def\l{\ifmmode\ell\else\polishlcross\fi}

\def\H{\ifmmode{H}\else\accent"07D\fi}

\DeclareMathOperator{\NU}{NU}
\DeclareMathOperator{\TSI}{TSI}

\DeclareMathOperator{\SL}{SL}
\DeclareMathOperator{\dist}{dist}

\DeclareMathOperator{\PI}{PI}
\DeclareMathOperator{\DL}{DL}
\DeclareMathOperator{\ret}{Ret}

\Large

\newcommand\incomp{\parallel}  
\usepackage{dsfont}
\newcommand{\zero}{\mathbf{0}}
\newcommand{\unit}{\mathbf{1}}

\usepackage{bm}

\newcommand{\pid}[1]{\bm{\langle} #1 ]}
\newcommand{\pfi}[1]{[ #1 \bm{\rangle}}

\newcommand{\ce}[2]{{#1}^{(#2)}}   
\newcommand{\cesa}{\ce\alpha{i}}
\newcommand{\cesb}{\ce{(\beta + 1)}j}
\newcommand{\ji}[2]{{#1}_{[#2]}}   
\newcommand{\jis}{\ji{\alpha}{i}}   
\newcommand{\mi}[2]{{#1}^{[#2]}}   
\newcommand{\mis}{\mi\beta{j}}   
\newcommand{\nVB}[4]{[\ji{#1}{#2}, \mi{#3}{#4}]} 
\newcommand{\nVBs}{\nVB{\alpha}{i}{\beta}{j}} 
\newcommand{\nEB}[4]{\pfi{\ji{#1}{#2}} \times \pid{\mi{#3}{#4}}} 
\newcommand{\nEBs}{\nEB{\alpha}{i}{\beta}{j}}
\newcommand{\UC}{\bigcup \cC}

\begin{document}
\title{Distributive lattice polymorphism on reflexive graphs}
\author{Mark Siggers}
\address{Kyungpook National University, Republic of Korea}
\email{mhsiggers@knu.ac.kr}
\thanks{Supported by the National Research Foundation (NRF) of Korea
  (2014-06060000)} 
\date{\today}


\begin{abstract}
  In this paper we give two characterisations of the class of reflexive graphs admitting
  {\em distributive lattice polymorphisms} and use these characterisations to address the 
  problem of recognition:  for a reflexive graph $G$ in which no two vertices have the same 
  neighbourhood, we find a polynomial time algorithm to decide if $G$ admits a 
  distributive lattice polymorphism. 
    
\end{abstract} 


\maketitle

\section{Introduction}\label{sect:results}

  \subsection{Motivation}
  It is well known (see for example \cite{BJK}) that Constraint Satisfactions Problems,
  which provide a  formulation for many combinatorial problems, can be stated as the problem 
  of finding a homomorphism between structures. Moreover (\cite{FV98}) any such homomorphism
  problem can be reduced to the retraction problem for {\em reflexive} graphs- graphs in
  which every vertex has a loop. 

  The problem $\ret(G)$ of retraction to $G$ is $NP$-complete for most reflexive graphs $G$ and
  in the case that the problem is known to be polynomial time solvable for some
  $G$, the polynomial time algorithm is tied to the existence of a {\em polymorphism} on 
  $G$--  an operation $f:V(G)^d \to V(G)$ that preserves edges--  satisfying some nice
  identity.
 
  A reflexive graph $G$ is a {\em lattice} graph if 
  it has a {\em compatible lattice}, a lattice on its
  vertex set such that the meet and join operations, $\wedge$ and $\vee$, 
  are polymorphisms of $G$.  It is a {\em distributive lattice graph} or {\em $\DL$-graph}
  if the lattice $L$ is distributive; we call $(G,L)$ a {\em $\DL$-pair}.  

  It was shown in \cite{CDK} that $\ret(G)$ can be
  solved, for a structure $G$, by a linear monadic Datalog program with at most  one extensional predicate per
  rule, (i.e., $G$ has caterpillar duality), if and only if it is a retract of a 
  $\DL$-graph. 

  To give a bit more context, a $d$-ary polymorphism $f:G^d \to G$ is a
  {\em totally symmetric idempotent} ($\TSI$) polymorphism if   
  $f(v,v,\dots, v) = v$ for all
  $v \in V(G)$ and if $f(v_1, \dots, v_d) = f(u_1, \dots, u_d)$ whenever
  $\{v_1, \dots, v_d\} = \{u_1, \dots, u_d\}$.
  The class $\TSI$ of reflexive graphs $G$ admitting $\TSI$ polymorphisms of all aritites is 
  important, as this is
  the class for which $\ret(G)$ can be solved by a monadic Datalog program
  with at most one extensional predicate per rule,
  (i.e., $G$ has tree duality).

  It is of interest to get a graph theoretic characterisation of the class $\TSI$. 
  The two main sources of $\TSI$ polymorphisms, are near-unanimity ($\NU$) polymorphisms, 
  and semilattice ($\SL$) polymorphisms. While $\NU$ polymorphisms have been well studied,
  and the classes of graphs admitting them have several nice characterisations,
  no such study had been attempted for $\SL$ polymorphisms until \cite{HS}, where we
  looked at the family of reflexive graphs admitting $\SL$ polymorphisms.

  A (meet) semilattice ordering on the vertices of a graph defines a $2$-ary operation 
  $\vee$ on the vertex set. If this operation is a polymorphism of the graph then it is
  an {\em $\SL$ polymorphism} of the graph.   
  The problem of characterising reflexive graphs admitting $\SL$ polymorphsims was 
  difficult, and we restricted our attention to those graphs $G$ that admit $\SL$
  polymorphisms for which the Hasse diagram of the semilattice ordering is a tree, and 
  a subgraph of $G$.  We showed that the class of such graphs extends the class of 
  chordal graphs.  We were unable to say much in the case that the semilattice was not
  a tree. The other extreme is when the ordering is a lattice-- the Hasse diagram is not
  a tree except in the  trivial case that the lattice it is a chain.  This leads us to 
  consider lattice polymorphisms.     
   
  \subsection{Results}

  In this paper we give two explicit characterisations of the class of reflexive 
  $\DL$-graphs, and use these characterisations to address the problem of 
  recognition. 

   For our  first characterisation,
   we recall a well known result of Birkhoff \cite{Bi}. 
   For a poset $P$, a subset $D$ is a {\em downset} if
   $b \in D$ and $a \leq b$ implies $a \in D$. The family $\cD(P)$ of all downsets of $P$
   is a distributive lattice under the ordering $\subseteq$. 
   The meet and join operations are $\cap$ and $\cup$, respectively.
   Birkhoff showed that for any distributive lattice $L$, $L$ is isomorphic
   to $\cD(J_L)$ for a unique poset $J_L$, (the poset of join irreducible elements of $L$).    
  
   Viewing a comparability $a \leq b$ as an arc $(a,b)$, a poset $P$ is just
   a transitive acyclic (except for loops) reflexive digraph.
   So we can talk of a sub-digraph $A$ of $P$.   
 
  \begin{definition}[$G(P,A)$]\label{def:GPA}
      For a poset $P$ and a sub-digraph $A$ of $P$, let $G = G(P,A)$ be the graph 
      on $\cD(P)$, in which two downsets $D, D' \in \cD(P)$ are adjacent if
      $A$ contains all arcs $(x,y)$ of $P$ for which $x$ and $y$ are in either
     $D \setminus D'$ or $D' \setminus D$.
  \end{definition}

  See Figure \ref{fig:GPA} for an example. The left side shows a poset $P$ represented by its
  Hasse diagram (defined in Section \ref{sec:def}) in thick light edges, and a sub-digraph
  $A$ in dark edges, missing only the arc $(b,c)$. 
  On the right is the downset lattice $\cD(P)$ again represented by its Hasse diagram, and
  the graph $G(P,A)$ missing only edges between vertices one of which contains $b$ and $c$  and 
  the other of which contains neither of them.

  \begin{figure} 
   \begin{center}
    \begin{picture}(8,5)(0,0)%
    \put(0,0){\includegraphics[width=8cm]{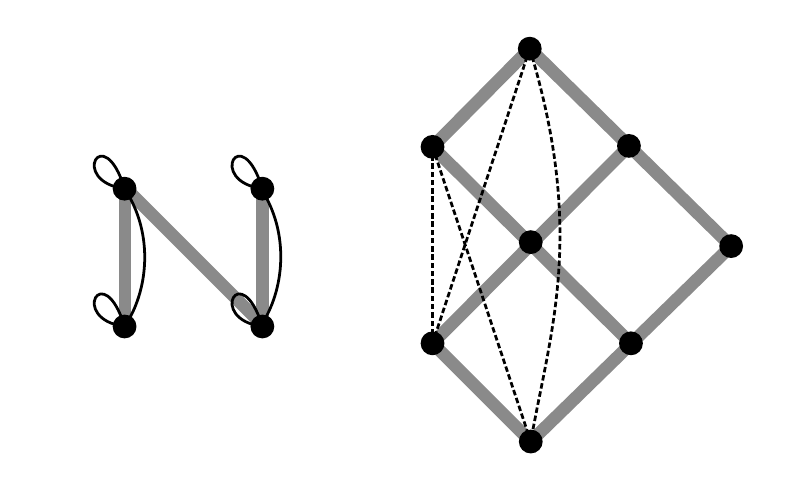}}%
    \put(1.3,0){\smash{$P$ and $A$}}%
    \put(1.1,1.3){\smash{$a$}}%
    \put(2.6,1.3){\smash{$b$}}%
    \put(1.3,3.2){\smash{$c$}}%
    \put(2.7,3.2){\smash{$d$}}%
    \put(4.2,0){\smash{$\cD(P)$ and $\overline{G(P,A)}$}}%
    \put(5.6,0.4){\smash{$\emptyset$}}%
    \put(4.2,1.1){\smash{$a$}}%
    \put(6.3,1.1){\smash{$b$}}%
    \put(5.8,2.4){\smash{$ab$}}%
    \put(7.6,2.4){\smash{$bd$}}%
    \put(4.0,3.7){\smash{$abc$}}%
    \put(6.3,3.7){\smash{$abd$}}%
    \put(5.2,4.7){\smash{$abcd$}}%
  \end{picture}%
\end{center}
\caption{Poset $P$ and lattice $\cD(P)$ in thick light edges. Digraph $A$ and (the complement of) graph $G(P,A)$ in
     dark. }
  \label{fig:GPA}
  \end{figure}
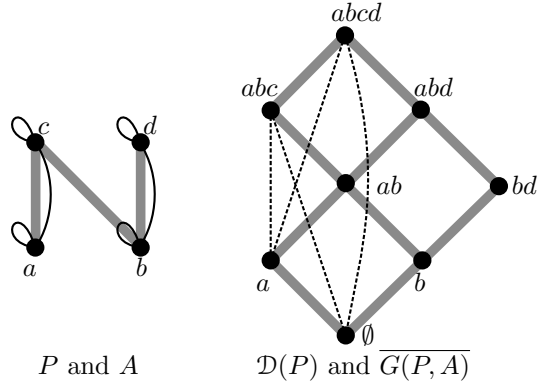

  Our first main theorem, Theorem \ref{thm:char2} says that for any $\DL$-pair $(G,L)$, 
  $G$ is isomorphic to $G(J_L,A)$ for some sub-digraph $A$ of $J_L$.
  Showing, in Lemma \ref{lem:pographiscompat}
  that $G(P,A)$ is always a $\DL$-graph, we get the following characterisation of
  reflexive $\DL$-graphs.    

  \begin{corollary}\label{cor:char2} 
     A reflexive graph $G$ is a $\DL$-graph if and only if there is 
     there is a poset $P$ and a sub-digraph $A$ such that 
     $G \iso  G(P, A)$.
  \end{corollary}
 
  In \cite{Ga} it was shown that a graph (without loops) is a {\em proper interval graph} if and
  only if there is an ordering of its vertices such that if it satisfies the so-called
  {\em min-max} identity:
  \begin{equation}\label{id:minmax}
   (u' \leq u \leq v \leq v' \mbox{ and } u' \sim v')  \Rightarrow u \sim v
 \end{equation}
  We take this as our definition of a  proper interval graph in the reflexive context, and say
  a proper interval graph is in {\em min-max form} if its vertices are labelled $\{0,1, \dots, n\}$
  for some $n$ so that it satisfies \eqref{id:minmax}.

  Simple arguments (see Fact \ref{fact:ids}) show that any reflexive graph
  that is compatible with a chain lattice is a proper interval graph.
  As any distributive lattice is embeddable in a product of chains, 
  it follows that any $\DL$-graph is a subgraph of a categorical product
  (defined in Section \ref{sec:def}) of proper interval graphs. 
  In fact Dilworth \cite{Di50} showed, and we recall this in more detail in Section \ref{sec:main2},
  that any chain decomposition of $P$ yields an embedding of $\cD(P)$ into a product of chains.
  The embedding shown in Figure \ref{fig:GPA} comes from the decomposition of $P$ into the chains
  $a\prec c$ and $b \prec d$. Figure \ref{fig:PI} shows the embedding corresponding to the 
  decomposition of $P$ into the chains $a$, $b \prec c$, and $d$.  
  For the embedding in Figure \ref{fig:PI} the graph $G(P,A)$ is an induced subgraph of proper interval
  graphs, in this case paths, on the chain factors. It turns out that this happens when certain edges
  in the complement of $A$ in $P$ are contained in the chain decomposition of $P$.

  For any $\DL$-pair $(G,L)$, we get, in Theorem \ref{thm:embedding},   
  an embedding of $L$ into a product of chains such that $G$ is an induced subgraph of $\cG$. 
  Further, it is an induced subgraph of quite a particular form. 
  
  The vertex set of a product $\cG = \prod_{i = 1}^d G_i$ of proper interval graphs $G_i$ is a set
  of $d$-tuples $x = (x_1, \dots, x_d) \in \prod_{i = 1}^d \{0,1, \dots, n_i\}$ for some 
  $n_1, \dots, n_d$.
  A {\em vertex interval} of $\cG$ is the set 
  \[  \nVBs = \{ x \in V(\cG) \mid  \alpha \leq x_i \mbox{ and } x_j \leq \beta \} \]
  for some $i,j \in [d]$, $\alpha \leq n_i$ and $\beta \leq n_j$. (See right side of 
  Figure \ref{fig:PI}.)

  \begin{figure} 
   \begin{center}
    \begin{picture}(14,5)(0,0)%
    \put(6,0){\includegraphics[width=8cm]{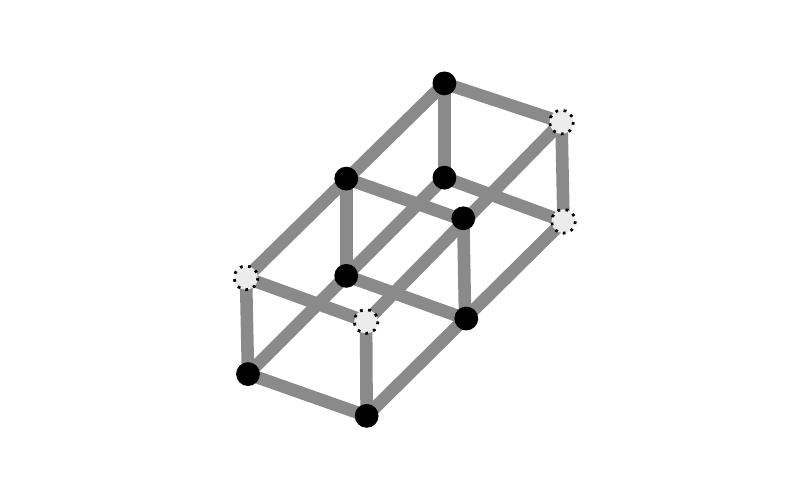}}%
    \put(8.1,-0.1){\smash{$\cP - \nVB1201 - \nVB2103$}}%
    \put(9.7,.4){\smash{$(0,0,0)$}}%
    \put(10.9,1.5){\smash{$(1,0,0)$}}%
    \put(11.9,2.6){\smash{$(2,0,0)$}}%
    \put(11.9,3.7){\smash{$(2,1,0)$}}%
    \put(8, .7){\smash{$(0,0,1)$}}%
    \put(10, 4.4){\smash{$(2,1,1)$}}%
    \put(7.2, 2.0){\smash{$(0,1,1)$}}%
    \put(7.2, 2.0){\smash{$(0,1,1)$}}%
    \put(0,0){\includegraphics[width=8cm]{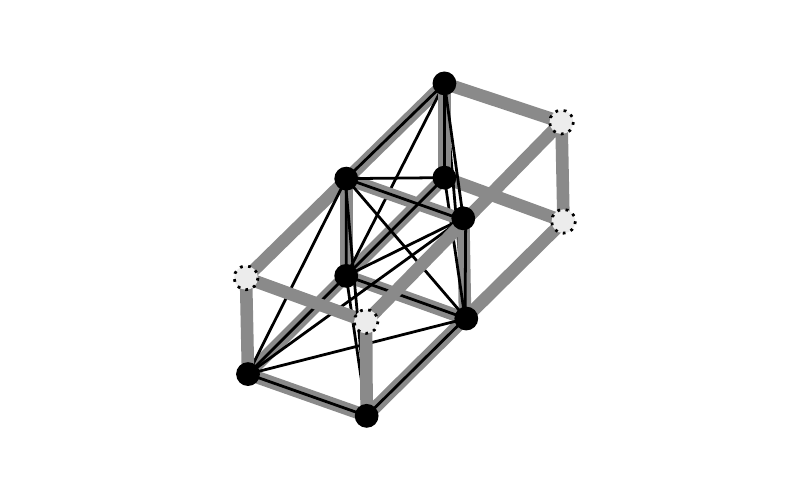}}%
    \put(2.7,-0.1){\smash{$\cD(P)$ and $G(P,A)$}}%
    \put(4,.5){\smash{$\emptyset$}}%
    \put(4.9,1.5){\smash{$b$}}%
    \put(5.9,2.5){\smash{$bc$}}%
    \put(5.9,3.7){\smash{$bcd$}}%
    \put(3,2.1){\smash{$ab$}}%
    \put(2.8,3.1){\smash{$abd$}}%
    \put(4.6,3.25){\smash{$abc$}}%
    \put(4.8,2.5){\smash{$bd$}}%
    \put(2.3, .9){\smash{$a$}}%
    \put(2.1, 2.1){\smash{$d$}}%
    \put(4, 4.4){\smash{$abcd$}}%
   
  \end{picture}%
\end{center}
\caption{Left: The lattice $\cD(P)$ from Figure \ref{fig:GPA} embedded  
          in a product of three chains, and the
          graph $G(P,A)$ from Figure \ref{fig:GPA} embedded as an induced
          subgraph of the product of paths on those chains. Right: The usual labelling
          on the product of chains showing $\cD(P)$ as $\cP - \nVB1201 - \nVB2103$.}      
  \label{fig:PI}
  \end{figure}
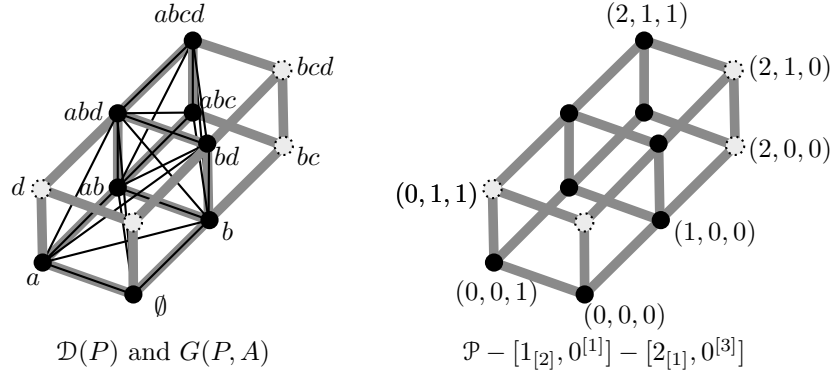

  Our following, second, characterisation of reflexive $\DL$-graphs is immediate from
   Theorem \ref{thm:embedding}.  
 
  \begin{corollary}\label{cor:charPI}
   A reflexive graph $G$ is a $\DL$-graph if and only if it is the induced subgraph of a 
   product $\cG = \prod_{i = 1}^dG_i$ of proper interval graphs $G_i$ in min-max form, 
   that we get by removing vertex intervals. 
  \end{corollary}


   A reflexive graph $G$ is $R$-thin if no two vertices have the same neighbourhood. 
   For questions of $\ret(G)$, one may always assume that $G$ is $R$-thin as there are
   simple linear time reductions between $\ret(G)$ and $\ret(G^R)$ where $G^R$, defined formally in
   Section \ref{sub:NRThin}, is the $R$-thin graph we get from $G$ by removing all but one
   vertex from every set of vertices sharing the same neighbourhood. 
   In Section \ref{sect:recog} we prove the following.

   \begin{theorem}\label{thm:poly}
     There is a polynomial time algorithm to decide whether
     or not an $R$-thin reflexive graph is a $\DL$-graph. 
   \end{theorem}    

   In fact, we give a polynomial time algorithm that not only decides if
   a given $R$-thin graph has a compatible distributive lattice, but if it does,
   finds one (actually all) such lattices.  

   While for questions about $\ret(G)$ one may assume that $G$ is $R$-thin, distributive lattice polymorphisms
   are unusual in the fact that the existence of a compatible distributive lattice for $G^R$ does not imply the existence of one for $G$. 
   We finish of Section \ref{sect:recog} with some notes about deciding if a non $R$-thin graph is a $\DL$-graph.

   
 



   \section{Definitions, Notation, and Basic Observations}\label{sec:def}  
    
    For any element $u$ of any ordering,
    $\pid{u}$ is the set of elements below $u$, and $\pfi{u}$ is the set of elements above it.  
    We write $a \prec b$ if $b$ {\em covers} $a$;
   that is, if $a < b$ and  there is no $x$ such that $a < x < b$. It is standard to 
   depict a poset by it's Hasse diagram-- its sub-digraph of covers-- and to 
   depict direction of the covers simply by assuming that the greater element is higher on the page.
   As we draw a lattice and a graph on the same set of vertices, the edges of our
   Hasse diagram are the thicker lighter edges, and the graph edges are thin
   and dark. 

    Recall that a {\em lattice} $L$ is a partial ordering on a set such that the 
    greatest lower bound and least upper bound are uniquely defined for any pair
    of elements.  These define the meet, $\wedge$,  and join, $\vee$, operations 
    respectively. It is a basic fact that the operations $\wedge$ and $\vee$ and the 
    lattice defined each other by the identities  
    \[ u \leq v \iff (u \wedge v) = u \mbox{ and } u \leq v \iff (u \vee v ) = v. \]  
    The lattice is {\em distributive} if the meet and join distribute. As our lattices are
    finite the meet and join operations are well defined for any set of elements and 
    there is a  maximum element, or {\em unit},  denoted $\unit$, and a minimum element
    or {\em zero} denoted $\zero$.

    A lattice $L$ on the vertices of a graph $G$ was defined to be compatible if its
    meet and join operations $\wedge$ and $\vee$ are polymorphisms. 
    Explicitly, $L$ is compatible with $G$ if and only if the following holds,
    where `$\sim$' denotes adjacency in $G$:
    \begin{equation}\label{id:poly}
    ( u \sim u' \mbox{ and }  v\sim v') \Rightarrow 
     (u \wedge v \sim u' \wedge v' \mbox{ and } u \vee v \sim u' \vee v'). 
    \end{equation}

     Along with \eqref{id:minmax}, there is another useful property of an ordering of vertices.   
     \begin{equation}\label{id:vee} 
   u \sim v \sim w  \mbox{ and } (u \leq v \geq w \mbox{  or } u \geq v \leq w)
   \Rightarrow u \sim w  
     \end{equation}

    \begin{fact}\label{fact:ids}
     For a compatible lattice ordering of a reflexive graph $G$ 
     identities \eqref{id:minmax} and \eqref{id:vee} hold for all vertices
     $u,v,w$ of $G$. 
   \end{fact}
     
    \begin{proof} For \eqref{id:minmax}, as $v \sim v$ we get 
     $u = u \wedge v \sim w \wedge v = v$ and $v = u \vee v \sim w \vee v = w$,
     as needed. 
     For \eqref{id:vee}, assuming that $u \leq v \geq w$, we get
     $u = (u \wedge v) \sim (v \wedge w) = w$. 
    \end{proof}





  The product $L_1 \times L_2$ of two lattices is the ordering on the set 
  $L_1 \times L_2$ defined by 
     \[ (a_1,a_2) \leq (b_1, b_2) \mbox{ if } a_i \leq_i b_i \mbox{ for } i = 1,2, \]
  and the operations $\vee$ and $\wedge$ of the product are defined
  componentwise from the corresponding operations of the factors. 
  Thus the product of distributive lattices is a distributive lattice.
  The {\em (categorical) product} of two graphs $G_1$ and $G_2$,
  is the graph $G = G_1 \times G_2$ with 
  vertex set $V(G_1) \times V(G_2)$ and edgeset
    \[ \{ (u_1,u_2)(v_1,v_2) \mid u_iv_i \in G_i \mbox{ for } i = 1,2\}. \] 

  The following is standard. 
  
    \begin{lemma}\label{lem:product} If the reflexive graph $G_i$ is compatible with 
     the lattice $L_i$ for $i = 1,2$ then $G_1 \times G_2$ is compatible with 
     $L_1 \times L_2$. 
    \end{lemma}
     \begin{proof}
     Let $(u_1,u_2)\sim (v_1,v_2)$ and $(u'_1,u'_2)\sim (v'_1,v'_2)$ in
     $G_1 \times G_2$. 
     Then 
      \[ (u_1,u_2) \wedge (u'_1,u'_2) = (u_1 \wedge u'_1, u_2 \wedge u'_2) \sim
           (v_1 \wedge v'_1, v_2 \wedge v'_2) =  (v_1,v_2) \wedge (v'_1,v'_2),\]
     and similarly $(u_1,u_2) \vee (u'_1,u'_2) =  (v_1,v_2) \vee (v'_1,v'_2)$. 
   \end{proof}


    A sublattice $L'$ of a lattice $L$ is any subset that is closed under the meet
    and join operations. The following is clear from the definition of compatibility. 
    
   \begin{fact}\label{fact:sublat}
    If a graph $G$ is compatible with a lattice $L$, and $L'$ is a sublattice of $L$,
    then the subgraph $G'$ of $G$ induced by $L'$ is compatible with $L'$.
   \end{fact} 

     A {\em conservative set} (or subalgebra) in a reflexive graph $G$ is an subset
     $S \subset V(G)$ that is the intersection of sets of the form
     $\{ x \in V(G) \mid d(x,x_0) \leq d \}$ for some
     vertex $x_0$ and integer $d$.  Components and maximal cliques are examples
     of conservative sets.  It is a basic fact, (see \cite{BJK}), that a conservative
     set of a graph is closed under any polymorphism. We use this to prove the following,
     which allows us to restrict our attention to connected graphs. 

   \begin{lemma}\label{lem:connected} 
     A graph is a (distributive) lattice graph if and only if each component is. 
   \end{lemma}
   \begin{proof} 
    If a graph is disconnected, and each of its components has a compatible 
    lattice $L_i$, 
    then let $L$ be the {\em simple join} of the component lattices;  
    that is, let $L$ be the lattice on the set $\bigcup_{i = 1}^d L_i$ with the 
    ordering defined by $x \leq y$ if $x \leq y$ in some $L_i$ or if $x \in L_i$
     and $y \in L_j$ for 
     $i < j$. It is easy to check that this lattice is compatible with $G$, 
    and that it is distributive
    if the component lattices are. 
    
   On the other hand, if a disconnected graph has a compatible lattice, 
   then as each component is a subalgebra, and subalgebras are closed under
   polymorphisms,
   each component is closed under the lattice operations. 
   Thus each component induces a
   sublattice, so is compatible with the component by Fact \ref{fact:sublat}. 
   If a lattice is distributive, then so is any sublattice.   
   \end{proof}





   The following, which does not hold for semilattices, is a huge simplifaction. 

  \begin{proposition}\label{prop:HasseSubgraph}
    For a connected reflexive graph $G$ with a compatible lattice $L$,
    the Hasse diagram of $L$ is a subgraph of $G$. 
  \end{proposition}
  \begin{proof}
   It is enough to show for any cover $v \lcov u$, that $uv$ is an edge of $G$. 

   Observe first that the upset $\pfi{v}$  is a connected subgraph of $G$.
   Indeed as $G$ is connected, for $u_0$ and $u_p$ in $\pfi{v}$,
   there is   a path  $u_0 \sim u_1 \sim \dots \sim u_p$ between them in $G$. 
   So $(v \vee u_0) \sim (v \vee u_1) \sim \dots \sim (v \vee u_p)$ 
   is a walk between them in $\pfi{v}$.
   
   The same proof in connected $\pfi{v}$ then shows that the downset
   $\pid{u}$ in $\pfi{v}$ is
   connected.  But it  contains only $u$ and $v$, so $uv$ is an edge.  
  \end{proof}

  \section{Some Examples}\label{sect:examples}

   As all but the minimum and maximum vertex of a lattice must have at least one
   cover and be covered by one other vertex, the following is immediate from 
   Proposition \ref{prop:HasseSubgraph}.

  \begin{example}\label{ex:degOne}
    For a connected reflexive graph $G$ with a degree one vertex $v$,
    $v$ must be the minimum or maximum vertex of any compatible lattice $L$. 
    In particular, the only reflexive trees with compatible lattices are 
    paths.  
  \end{example}

  \begin{figure} \setlength{\unitlength}{340.1574707bp}
    \begin{picture}(1,0.43333335)(-.2,0)%
    \put(0,0){\includegraphics[width=\unitlength]{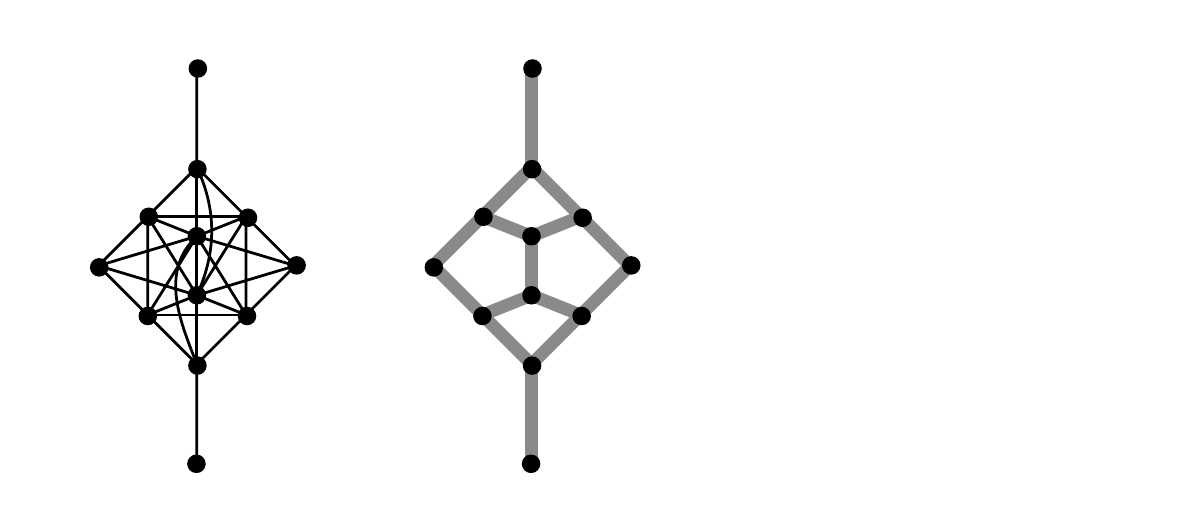}}%
    \put(0.46666669,0.02500007){\color[rgb]{0,0,0}\makebox(0,0)[lb]{\smash{$0$}}}%
    \put(0.46666669,0.36666671){\color[rgb]{0,0,0}\makebox(0,0)[lb]{\smash{$1$}}}%
    \put(0.45833336,0.10833334){\color[rgb]{0,0,0}\makebox(0,0)[lb]{\smash{$a$}}}%
    \put(0.38553573,0.14279105){\color[rgb]{0,0,0}\makebox(0,0)[lb]{\smash{$b$}}}%
    \put(0.49388894,0.14195102){\color[rgb]{0,0,0}\makebox(0,0)[lb]{\smash{$c$}}}%
    \put(0.45693121,0.18394846){\color[rgb]{0,0,0}\makebox(0,0)[lb]{\smash{$d$}}}%
    \put(0.45273149,0.23854504){\color[rgb]{0,0,0}\makebox(0,0)[lb]{\smash{$e$}}}%
    \put(0.34437829,0.19822752){\color[rgb]{0,0,0}\makebox(0,0)[lb]{\smash{$f$}}}%
    \put(0.54596561,0.20326726){\color[rgb]{0,0,0}\makebox(0,0)[lb]{\smash{$g$}}}%
    \put(0.3905754,0.26206356){\color[rgb]{0,0,0}\makebox(0,0)[lb]{\smash{$h$}}}%
    \put(0.50396828,0.25366398){\color[rgb]{0,0,0}\makebox(0,0)[lb]{\smash{$i$}}}%
    \put(0.4636508,0.29230159){\color[rgb]{0,0,0}\makebox(0,0)[lb]{\smash{$j$}}}%
  \end{picture}%
  \caption{Graph (left) with compatible lattice (right) but no compatible distributive lattice}
  \label{fig:LnotDL}
  \end{figure}
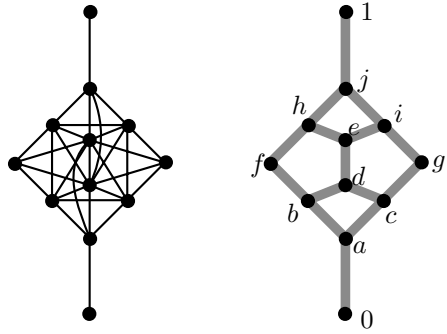

  \begin{proposition}\label{prop:badretracts}
    Neither the class of graphs admitting compatible lattices, nor the class 
    admitting compatible distributive lattices, are closed under retraction.  
  \end{proposition}
  \begin{proof}
    It is easy to see that the reflexive biclique $K_{1,4}$ is a retract of the 
    product $P_2^2$ of two reflexive paths. $P_2^2$ has a distributive lattice
    by Lemma \ref{lem:product}, but $K_{1,4}$ does not, by Example \ref{ex:degOne}. 
  \end{proof}

  \begin{proposition}\label{ex:bigone}
    There are graphs that have compatible lattices but have no compatible distributive 
    lattices. 
  \end{proposition} 
  \begin{proof}
    Let $G$ be the graph on the left of Figure \ref{fig:LnotDL}.  It is 
    easy but tedious to verify, using \eqref{id:poly} that the non-distributive
    lattice shown on the right is compatible.  We show that there is no distributive
    lattice that is compatible with $G$. 

    Assume, towards contradiction, that $G$ has a compatible distributive lattice.
    By Proposition \ref{prop:HasseSubgraph},  $\zero$ and $\unit$ must be the 
    vertices labelled $0$ and $1$ in the figure. Further $\zero$ must have
    unique cover $a$ and  $\unit$ must cover $j$. So $\pid{j} \cap \pfi{a}$
    is a distributive 
    sublattice with zero $a$ and unit $j$.  

    As  the set $\{d,e\}$ is the intersection of maximal cliques, it is a conservative 
    set, so induces a sublattice. The only $2$ element lattices is the chain, so we may
    assume, without loss of generality, that $d \leq e$.  

    The set $\{d,e,h\}$ is also an intersection of maximal cliques, so induces a sublattice
    of three elements, so must also be a chain.  If $h \leq e$ then 
    by \eqref{id:vee} $a$ and $h$ are adjacent, so $h \geq e$. 
     Similarly $b \leq d$.

    The set $\{b,d,e,f,h\}$ is a maximal clique, so induces a lattice. As $b$ is not adjacent
    to $a$ or $j$, it follows from \eqref{id:vee} that it can neither be above 
    or below $d$ or $e$, so it is incomparable with them.  Thus the sublattice
    induced on $\{b,d,e,f,h\}$ is as shown in the figure.
    It is well known that no lattice with 
    this lattice as a sublattice is distributive. 
  \end{proof}


   \section{Downset Construction}\label{sec:downsets}  

  Our main result of this section is Theorem \ref{thm:char2}.  
  Before we prove it, we make some easy observations about
  the construction $G(P,A)$ of Definition \ref{def:GPA}. Recall that $A$ is a sub-digraph of
  a poset $P$; it will always have the same vertex set as $P$
  We write
  $a \to b$ to mean that $(a,b)$ is an arc of $A$.

  
    An arc $(x', y')$ of $P$ is {\em useless} in a sub-digraph $A$
    if there is $(x,y) \not\in A$ with $x' \leq x \leq y \leq y'$ (and
    $x' \neq x$ or $y' \neq y$). Removing all useless arcs from $A$ it clearly 
    satisfies the following directed version of \eqref{id:minmax}
   \begin{equation}\label{id:minmaxd}
     (u' \leq u \leq v \leq v' \mbox{ and } u' \to v')  \Rightarrow u \to v
   \end{equation}

    \begin{lemma}\label{lem:extension1}
      The graph $G(P,A)$ is unchanged by adding or removing useless arcs from $A$.
      Thus $A$ may be assumed to satisfy \eqref{id:minmaxd}.
    \end{lemma} 
    \begin{proof}
        Let $x' \leq x \leq y \leq y'$, and $x \not\to y$ in $A$.
        For any two downsets $D$ and $D'$ of $P$ with $x',y'$ in $D \setminus D'$, we 
        clearly have that $x,y$ are in $D \setminus D'$ as well, and so $D \not\sim D'$ in 
        $G(P,A)$ whether $x' \to y'$ or not.  
    \end{proof}

   For a sub-digraph $A$ of a poset $P$, let $\comp{A}$ be the sub-digraph of $P$ with
   arc set
      \[ \{ (x, y)  \mid (x,y) \in P \setminus A \}.\]  
   The following is a useful alternate definition of adjacency in $G(P,A)$.

  \begin{lemma}\label{lem:altGPA}
    Where $A$ is a sub-digraph of a poset $P$, and $D$ and $D'$ are in $\cD(P)$, 
    $D$ and $D'$ are adjacent in $G(P,A)$ if and only if the following hold for all
    vertices $x$ and $y$.
     \begin{itemize}
       \item  If $x \in D$ and $(y,x) \in \comp{A}$ then $y \in D'$, and
       \item  If $x \in D'$ and $(y,x) \in \comp{A}$ then $y \in D$.
       \end{itemize}  
   \end{lemma}  
  \begin{proof}
    The definition of adjacency of $D$ and $D'$ is clearly equivalent to the 
    statement that  neither of $D \setminus D'$ or $D' \setminus D$ induce an
    edge of $\comp{A}$. That $D \setminus D'$ induces no edge in $\comp{A}$ is equivalent
    to the statement that for all $(y, x)$ in $\comp{A}$ with $x,y$ in $D$, 
    either $y \not\in D'$ or $x \not\in D'$. As $D$ and $D'$ are downsets, this reduces
    to the statement that for all $(y, x)$ in $\comp{A}$ with $x \in D$, $y \not\in D'$.
  \end{proof}

 


   \begin{theorem}\label{thm:char2} 
    For any reflexive graph $G$ compatible with a distributive lattice $L$,
    $G \iso G(J_L, A)$ for a unique sub-digraph $A$ of $J_L$    
    satisfying \eqref{id:minmax}.   
   \end{theorem}
   \begin{proof}
     Let $G$ be compatible with a distributive lattice $L$. By \cite{Bi}
     we have that $L \iso \cD(J_L)$, so we denote vertices of $G$ by downsets of the 
     poset $J_L$. 

     We define a sub-digraph $A = A(G,J_L)$ of $J_L$ as follows. 
     For a vertex $p$ of $J_L$, let 
      \[ C_p = X \setminus \pfi{p} = \bigcup \{ X \in \cD \mid p \not\in X \} \] be the
     maximum downset not containing $p$.  
     For each arc $y \to x$ of $J_L$, let $y \to x$
     be in $A$ if $\pid{x} \cap C_y$ and $\pid{x}$ are adjacent in $G$.
     We show that $G \iso G(J_L, A)$. 
     
     Let $D$ and $D'$ be adjacent downsets of $J_L$. To show that $D$ and $D'$ are 
     adjacent in $G(J_L,A)$, it is enough to show, without loss of generality,
     that any arc $(y,x) \in J_L$ for $x,y \in (D \setminus D')$,
     is in $A$. So we must show that $\pid{x} \cap C_y$ and $\pid{x}$ are adjacent. 
     As $D' \sim D$ and $\pid{x} \sim \pid{x}$  in $G$ we have that 
        \[ (\pid{x} \cap D') \sim (\pid{x} \cap D) = \pid{x}. \]
     But as $y \not\in D'$ we have that $D' \leq C_y$, and so we also have
        \[ \pid{x}\cap D' \leq \pid{x} \cap C_y  \leq \pid{x}; \]
     and so by \eqref{id:minmax} we get  $\pid{x} \sim (\pid{x} \cap C_y)$, as
     needed. 
     
     On the other hand, let $D$ and $D'$ be non-adjacent downsets of $J_L$.
     Then we must show that there is some arc $(y,x) \in J_L$ for 
     $x,y \in (D \setminus D')$ or $x,y \in (D' \setminus D)$ that is not in
     $A$. Assume that all such arcs with $x,y \in (D \setminus D')$ are in $A$.    
     Then for each, we saw above that $\pid{x} \sim (\pid{x} \cap C_y)$. 
     Fixing $x$ and taking the intersection over all $y \leq x$ in $D \setminus D'$,
     we get
     \[ \pid{x} = \bigcap \pid{x} \sim \bigcap (\pid{x} \cap C_y) = \bigcap (\pid{x} \setminus \pfi{y})
       =:T_x \]
      where $T_x = \pid{x} \setminus \bigcup \pfi{y}$ is contained in 
    $\pid{x} \setminus D'$ as the union is over all $y \leq x$ that are
    in $D \setminus D'$. 
     Now taking the union over all $x \in D \setminus D'$ we get that
     \[ D = \bigcup \pid{x} \sim \bigcup T_x \subseteq \bigcup (\pid{x} \setminus D') \subseteq (D \setminus D'). \]
      By \eqref{id:minmax} we get that $D \sim D \setminus D'$.
       
     Similarily we get that $D' \sim (D \cap D')$.
     But as $D \cap D' \leq D, D'$ we get from \eqref{id:vee} that 
     $D \sim D'$, a contradiction.

     Now,  Lemma \ref{lem:extension1} allows us
       to assume that $A$ satisfies \eqref{id:minmax}. The uniqueness of $A$ then follows by 
       observing that $G(J_L,A')$ would be different for any other sub-digraph $A'$ of $J_L$
       satisfying \eqref{id:minmax}: this is simple let $(y,x)$ be an arc of $A'$ but not
       $A$. Then the edge $\pid{x} \sim \pid{y}$ is in $G(J_L,A')$ but not in $G(J_L,A)$. 
   \end{proof}

  Corollary \ref{cor:char2} is  immediate from  Theorem \ref{thm:char2}
  and the following lemma.   

 \begin{lemma}\label{lem:pographiscompat}
   If $P$ is a poset and $A \subseteq P$, then 
   $\cD(P)$ is compatible with $G = G(P,A)$.
 \end{lemma}
 \begin{proof}
   We use Lemma \ref{lem:altGPA} for the definition of adjacency in $G$.
   Assume that $D \sim D'$ and $E \sim E'$. 
   We must show that $(D \cup E) \sim (D' \cup E')$ and
   $(D \cap E) \sim (D' \cap E')$.
   For the former, let $x \in D \cup E$ and $(x, y)$ be in $\comp{A}$.
   Then $x \in D$ or $E$, 
   so as $D \sim D'$ and $E \sim E'$, we have that $y \in D'$ or $E'$.
   Thus $y \in D' \cup E'$. 
   That $x \in D' \cup E'$ implies $y \in D \cup E$ is the same, so
    $(D \cup E) \sim (D' \cup E')$. 
   The proof of the latter is similar.
 \end{proof}


     Now, consider $G(C,A)$ where $C$ is a chain. All downsets are of the form 
     $\pid{c}$ for some $c \in C$, or $\emptyset=:\pid{-1}$. 
     As we may assume that $A$ satisfies \eqref{id:minmax}, two downsets 
     $\pid{x}$ and $\pid{y}$, for $y \leq x$ are adjacent if and only if $(y+1, x)$
     is an arc of $A$. The following is then clear, and is the starting point of our 
     next characterisation of reflexive $\DL$-graphs. 

     \begin{fact}\label{fact:PI}
       For a chain $C$ and a sub-digraph $A$, $G(C,A)$ satisfies \eqref{id:minmax},
       so is a proper interval graph. 
     \end{fact}
     \begin{proof}
       Assume that $\pid{u} \subsetneq \pid{v} \subsetneq \pid{w} \mbox{ and } \pid{u} \sim \pid{w}$
       in $G(C,A)$. 
       So $u < v < w$ and $u+1 \to w$ in $A$. 
       As we may assume that $A$ satisfies \eqref{id:minmax} we have 
       $u+1 \to v$ and $v+1 \to w$ in $A$, and
       so $\pid{u} \sim \pid{v}$ and $\pid{v} \sim \pid{w}$.  
     \end{proof}

 \section{Reflexive $\DL$-graphs as subgraphs of products of Proper Interval Graphs}\label{sec:main2}

   In this section we prove Theorem \ref{thm:embedding} and give some related results.
    
   \begin{theorem}\label{thm:embedding} 
     For any $\DL$-pair $(G,L)$,  there is an embedding 
     $L \iso \cP \setminus \cV$,  where $\cV$ is a union of 
     intervals of the form $\nVBs$, of $L$ into a product $\cP = \prod P_i$ of chains. 
     Further, for each $P_i$ there is a proper interval graph $G_i$ compatible with 
     $P_i$ such that $G$ is the subgraph of $\cG = \prod_{i = 1}^dG_i$ induced by vertices in $L$.
   \end{theorem}
   
    Before proving this, we observe that this gives Corollary \ref{cor:charPI}.

    \begin{proof}[Proof of Corollary \ref{cor:charPI}]
     If $G$ is a $\DL$-graph then Theorem \ref{thm:embedding} gives us the necessary embedding
     of $G$ into  a product of proper interval graphs. 

     On the other hand, assume we get $G$ from a product $\cG = \prod G_i$ of proper interval
     graphs $G_i$ in min-max form by removing vertex intervals $\nVBs$. The ordering on the 
     $G_i$ is a chain lattice $P_i$, so induces on $V(\cG)$ a lattice $\cP = \prod P_i$.
     By a result in \cite{Ri74}, the subset induced by removing sets of the form $\nVBs$ is a
     sublattice. 
     By Fact \ref{fact:sublat} it is therefore compatible with the subgraph $G$ of $\cG$ that it
     induces.  
    \end{proof}

   We start with some results from the literature that will help us prove Theorem \ref{thm:embedding}. 
  \subsection{Setup for the proof of Theorem \ref{thm:embedding}} 

  The first statement of Theorem \ref{thm:embedding} is acually from \cite{Si15}. 
  We explain, as we will have to build on this. 
  Generalising the notion of a {\em chain decomposition} in which the subchains
  must be disjoint,  a {\em chain cover} of a poset $P$ is family
  $\cC = \{C_1, \dots, C_d\}$ of 
  subchains of $P$ such that every element of $P$ is in one chain.
  Given $\cC$, label the elements of $P$ so that the subchain $C_i$ is 
  $\ce{1}{i} \prec\dots \prec \ce{{n_i}}{i}$ for some $n_i$; if an element
  is in more than one chain, it gets more than one label.    
  It is clear that a downset $D$ of $P$
  is uniquely defined by 
  the tuple $e_\cC(D) = (x_1, \dots, x_d)$
  where $x_i = |D \cap C_i|$. (Note that $D$ is thus the downset
  generated by the set $\{\ce{x_i}i \mid i \in [d] \}$.) 

  As Dilworth \cite{Di50} observed in the case that $\cC$ is a decomposition, 
  we observed in \cite{Si15} that $e_\cC$ is in fact a lattice embedding of $\cD(P)$
  into the  product of chains
  $\cP_\cC = \prod_{i = 1}^d P_i$ where $P_i$ is the chain 
      $0 \prec 1 \prec \dots \prec {n_i}$
  with one more element than $C_i$. 
  Thus by Birkoff's  result from \cite{Bi}, every chain cover of $J_L$ gives an embedding 
  $e_\cC$ of $L$ as a sublattice of a product $\cP_\cC$ of chains.
  
  In (Corollary 6.6 of) \cite{Si15} we showed that every embedding of $L$ as a
  sublattice of a product of chains such that $L$ contains the zero and unit of 
  the product, is $e_\cC$ for some chain cover $\cC$ of $J_L$.    
  
  The following notation will also be useful, and explains the notation $\nVBs$. 
  Given a product of chains $\cP_\cC = \prod_{i = 1}^d P_i$, let 
   \[ \jis = (\underbrace{0,\dots,0,\alpha}_{i},0,\dots,0) \qquad \mbox{ and } \qquad \mis = (n_1, \dots,n_{j-1},\beta,n_{j+1}, \dots,n_d)\]
  for all $i,j \in [d]$ and $\alpha,\beta$ with $0 \leq \alpha \leq n_i$ and $0 \leq \beta \leq n_j$.
 
 %

   \subsection{Proof of Theorem \ref{thm:embedding}} 

   Let $(G,L)$ be a $\DL$-pair, so by Theorem \ref{thm:char2}, $G = G(J_L,A)$ for some
   sub-digraph $A$ of $J_L$, and let $\cC$ be a chain cover of $J_L$.
   For each chain $C_i \in \cC$ let $A_i$ be the subgraph of $A$ induced by the
   vertices of $C_i$. By Fact \ref{fact:PI} we have that $G_i = G(C_i,A_i)$,
   on the chain $P_i$ is a proper interval graph. 
   Let $\cG =  \prod G_i$ be the product of these proper interval graphs.
   The embedding $e_\cC: \cD(J_L) \to \cP_\cC$ embeds $V(G)$ as a subset of $V(\cG)$. 

   In \cite{Si15}, we observed the following, using a result of \cite{Ri74}.
   \begin{proposition}
    Where $L$ is the image of the embedding $e_\cC: \cD(J_L) \to \cP_\cC$, 
       \[ L = \cP_\cC \setminus \cV \] 
    where $\cV$ is the union, over comparable pairs $\cesb \leq \cesa$ in $J_L$, of the intervals
    \[  \nVB\alpha{i}\beta{j} =  \{ x \in \cP_\cC \mid \alpha \leq x_i \mbox{ and } x_j \leq \beta \}.\]
  \end{proposition}
  
   So any chain cover $\cC$ of $J_L$ yields an embedding of $L$ into $\cP_\cC$ as required by the
   first statement of Theorem \ref{thm:embedding}. To show that $G$ can be taken as an induced
    subgraph of $\cG$
   we need a similar statement about how the edges of $G$ relate to those of $\cG$. 
   We get this by looking at how our construction $G(J_L,A)$ behaves notationally under the
   embedding $e_\cC$.

    To finish proving Theorem \ref{thm:embedding} we must show that $e_\cC$ embeds $G$ as an
    induced subgraph of $\cG$. This is not true of every chain cover $\cC$, but every chain
    cover embeds it as a subgraph.
   
    For each $i,j\in [d]$ and $0 \leq \alpha \leq n_i$ and $0 \leq \beta \leq n_j$, let
      \[ \nEB{\alpha}{i}{\beta}{j} := \left\{ \{x, y\} \mid \alpha \leq x_i \mbox{ and }
           y_j \leq \beta \right\}. \]

   \begin{lemma}\label{lem:tight}
      Where $\cC$ is a chain cover of $J_L$, $G = G(J_L,A)$ is a subgraph of the
      product $\cG = \prod G_i$ of proper interval graphs $G_i = (C_i,A_i)$. 
      In fact $G = \cG \setminus \cV \setminus \cE$ where 
       \begin{gather*}
           \cV = \bigcup \{ \nVBs \mid \cesb \leq \cesa \in J_L \} \\
          \cE = \bigcup \{ \nEBs \mid \cesb \to \cesa \in \comp{A} \}.
       \end{gather*}
   \end{lemma}
   \begin{proof}
    Let $D_x$ and $D_y$ be downsets of $J_L$ and let $x = e_\cC(D_x)$ and $y = e_\cC(D_y)$.
   
    If $D_x \sim D_y$ we have in particular that for each $i$, 
    $\ce{x_i+1}{i} \to \ce{y_i}{i}$ if $x_i < y_i$ and 
    $\ce{y_i+1}{i} \to \ce{x_i}{i}$ if $y_i < x_i$. So $\ce{x_i}{i} \sim \ce{y_i}{i}$ in 
    $G_i(C_i,A_i)$. This shows that $G$ is a subgraph of $\cG$.  

    As we mentioned above, it follows from \cite{Si15} that $V(G) = V(\cG) \setminus \cV$
    so we are done with the following claim.

   \begin{claim*}
     Vertices $x$ and $y$ of $G$ are  non-adjacent in $G$ if and only if
     there is some arc $\cesb \to \cesa$ in $\comp{A}$ such that $\{x,y\}$ is in the set
     $\nEB{\alpha}{i}{\beta}{j}$  of edges of $\cG$.
   \end{claim*}  
   \begin{proof}\claimproof
    Let $D_x$ and $D_y$ be the downsets of $J_L$ for which $e_\cC(D_x) = x$ and $e_\cC(D_y) = y$.
    Then $D_x \not\sim D_y$ if and only if there is some $\cesb \to \cesa$ in $\comp{A}$
   with $\cesa,\cesb$ in $D_x \setminus D_y$ ( or $D_y \setminus D_x$, but wlog we assume the former).
    This is true if and only if
    \begin{equation}\label{edge1}
         y_i < \alpha \leq x_i \mbox{ and } y_j \leq \beta <  x_j.
    \end{equation}
      But since $\cesb \to \cesa$ in $\comp{A}$ we certainly have that $\cesb \leq \cesa$ in $J_L$, so
    $x$ and $y$ are not in $\nVB\alpha{i}\beta{j}$. This means that neither of 
    \begin{equation}\label{edge2}
         (\alpha \leq x_i \mbox{ and } x_j \leq \beta) \mbox{ or }  (\alpha \leq y_i \mbox{ and } y_j \leq \beta) 
    \end{equation}
     hold.  As \eqref{edge1} and the negation of \eqref{edge2} are logically equivalent to 
     \[  \alpha \leq x_i \mbox{ and }  y_j \leq \beta, \]
     we get the claim.
   \end{proof} 
  
   This completes the proof of the lemma.   
 \end{proof}

   Now many of the sets $\nEBs$ in the above lemma may not actually contain edges of 
   $\cG$, so it begs the question: when is $G$ an induced subgraph of $\cG$. 
   Clearly it is induced if and only if we can express it as $G = \cG \setminus \cV$,
   but this does not mean that $\comp{A}$ must be empty. 
   We have been using Lemma \ref{lem:extension1} to remove edges to $A$ and assume that 
   it satisfies \eqref{id:minmax}; we may also use it to add all useless edges, doing
   so the complement (in $J_L$) is a 'reduced' version of $\comp{A}$: a graph $\redcomp{A}$ that
   generates the usual $\comp{A}$ by composition with $J_L$.

   \begin{corollary}\label{cor:tight}
     Where $\cC$ is a chain cover of $J_L$, $G = G(J_L,A)$ is $\cG \setminus \cV$
     if and only if $\redcomp{A}$ is a sub-digraph of $\UC = \bigcup C_i$.
   \end{corollary} 
   \begin{proof}
    On the one hand, let $\redcomp{A}$ be a subgraph of $\UC$. We show that no edge of
    $\nEBs$, for any $\cesb \to \cesa$ in $\redcomp{A}$, is in $\cG$.
    Indeed, for any  $\cesb \to \cesa$ in $\redcomp{A}$ 
    there exist $k, \gamma$ and $\delta$ such
    that $\cesb = \ce{(\delta +1)}k$ and $\cesa = \ce\gamma{k}$,
    so  $\mis = \mi{\delta}k$ and $\jis = \ji{\gamma}k$ and so
     \[ \nEBs = \nEB{\gamma}{k}{\delta}{k}. \] 
    For any edge $\{x,y\}$ of $\nEB{\gamma}{k}{\delta}{k}$ we have  $\gamma \leq x$ and
    $y \leq \delta$, but $\{\ce{x}k,\ce{y}k\}$ is not in $G_i$ as 
    $\ce{(\delta +1)}k \not\to \ce\gamma{k}$ in $A_k$.

    On the other hand, assume that $\redcomp{A}$ is not a subgraph of $\UC$.
    Then there is some $\cesb \to \cesa$ in $\redcomp{A}$ such that 
    $\cesa$ and $\cesb$ are not both in $C_k$ for some  $k$.  
    We show that the projection of $\pid{\cesa} \not\sim \pid{\ce{\beta}j}$ onto any of the 
    $G_k$ is an edge of $G_k$, and so $\pid{\cesa} \not\sim \pid{\ce{\beta}j}$ is an edge of 
    $\cG$. 
    Indeed, $(\pid{\cesa} \setminus \cesa) \sim \pid{\ce{\beta}j}$, or otherwise
    there is an arc in $\comp{A}$ in $(\pid{\cesa} \setminus \cesa) \setminus \pid{\ce{\beta}j}$
    which would contradict the existence of $\cesb \to \cesa$ in the reduced $\redcomp{A}$. 
    For any $k$ such that $\cesa \not\in C_k$, 
    $(\pid{\cesa} \setminus \cesa) \sim \pid{\ce{\beta}j}$ projects onto
    the same edge in $G_k$ as does $\pid{\cesa} \not\sim \pid{\ce{\beta}j}$, implying that it is
    in $G_k$.  
    Similarily $\pid{\cesa} \sim   \pid{\ce{\beta}j} \cup \{\cesb\}$ is an arc showing that the
    projection of $\pid{\cesa} \not\sim \pid{\ce{\beta}j}$ onto $G_k$ is an edge of $G_k$ for
    any $k$ such that $\cesb \not\in C_k$.     
   \end{proof}

   \begin{proof}[Proof of Theorem \ref{thm:embedding}]
     Let $G$ be a reflexive $\DL$-graph and $L$ a  compatible distributive 
     lattice. Theorem \ref{thm:char2} provides us a sub-digraph $A$ of $J_L$ such that 
     $G \iso G(J_L,A)$. As we mentioned above, we have from \cite{Si15} that every
     chain cover $\cC$ of $J_L$ yields an embedding $e_\cC: L \iso \cP_\cC \setminus \cV$ into a 
     product of chains. By Corollary \ref{cor:tight} it is induced if and only if 
     $\redcomp{A}$ is a subgraph of $\UC$. We can assure this by taking every arc of $\redcomp{A}$
     as a two element chain in $\cC$ and then covering then rest of $J_L$ with one element chains.  
   \end{proof}

   \subsection{Tight embeddings}\label{sub:nottight} 

     Theorem \ref{thm:embedding} tells us that for every $\DL$-pair $(G,L)$  
     there is an embedding of $L$ into a product $\cP$ of chains $C_i$
     such that $G$ is an induced
     subgraph of the product $\cG$ of proper interval graphs $G_i$. 
     We simply refer to this as an {\em induced embedding} of $(G,L)$.   
     An embedding is {\em tight} if every cover of $L$ is a cover of $\cP$.
     Classical results of Birkoff and  Dilworth correspond tight embeddings of $L$ into
     products of chains to {\em chain decompositions} of $J_L$: chain covers $\cC$ consisting
     of disjoint chains. 
    By Lemma \ref{lem:tight} any $\DL$-pair $(G,L)$ has a tight embedding, but if $\redcomp{A}$
    has any vertices with in-degree or out-degree greater than $2$, then by \ref{cor:tight} this
    it is not induced.  In fact, we will see at the end of the next section that there are $\DL$-graphs
    $G$ such that there are no compatible lattices $L$ for which $(G,L)$ has a tight induced
    embedding.  This is why we had to consider non-tight embeddings, and why we wrote \cite{Si15}.

    \section{Recognition of $R$-thin $\DL$-graphs}\label{sect:recog}

    Recall that a graph is {\em $\R$-thin} if no two vertices have the same neighbourhood.
    As our graphs are reflexive neighbourhoods and closed neighbourhoods are the same thing.

    The factorization of a categorical product was shown to be unique (up to certain
    obviously necessary assumptions which include $R$-thinness)
    in \cite{DI70} by D\"orfler and Imrich.  Feigenbaum and Sch\"affer \cite{FS86} showed
    that a categorical product 
    can be factored in polynomial time.  On the other hand, it was shown in \cite{CKNOS}
    that proper interval graphs can be recognised
    in linear time.  So products $\cG$ of proper interval graphs can be recognised in 
    polynomial time.

   Using techniques discussed in \cite{HIK} we will prove Theorem \ref{thm:poly}, which says
   that there is a polynomial time algorithm to decide whether
     or not an $R$-thin reflexive graph is a $\DL$-graph.  We make no effort
   to optimise our algorithm.    
    First we develop some properties of $\DL$-graphs.

     \subsection{Tightest embeddings and $R$-thinness}

     As we mentioned in Subsection \ref{sub:nottight}, not all $\DL$-pairs $(G,L)$ have
     tight induced embeddings. We call an induced embedding {\em tightest} if
     it minimises then number of {\em non-tight covers}-- covers of $L$ that
     are not covers of $\cP$. We prove some properties of tightest embeddings.

     \begin{claim}\label{cl:6-1} 
      If $x \prec y$ is an non-tight cover in a tightest induced embedding of a $\DL$-pair $(G,L)$, 
      then for every vertex $z$ of $G$ either $z \leq x$ or $y \leq z$.
     \end{claim}
     \begin{proof}
       Indeed, $x < z < y$ is impossible as $x \prec y$. If $x < z$ but $z \incomp y$ (i.e., $z$ and $y$
       are incomparible) then $z_i > y_i$ for some $i$ and so $x_i < (z \wedge y)_i$ giving us that 
       $x < z \wedge y < y$, which is again impossible as $x \prec y$. 
       Similarily $x \incomp z$ and $x < y$ is impossible. Finally, if $x \incomp z$ and 
       $y \incomp z$ then taking $z' = x \vee z$ we get that $x < z'$ and $z' \incomp y$, 
       which we have already seen is impossible. 
     \end{proof}     
     
      \begin{claim}\label{cl:6-2} 
       If $x \prec y$ is an non-tight cover  in a tightest induced embedding of a $\DL$-pair $(G,L)$,
       then for all $i \in [d]$, $x_i$ and $y_i$ have different neighbourhoods in $G_i$. 
     \end{claim}    
     \begin{proof}
       Towards contradiction, assume that $x_i$ and $y_i$ have the same  neighbourhoods,
       then replacing $G_i$ with the  proper interval graph we get by contracting $x_i$ and $y_i$
       into a point, and reducing the $i$-coordinate of every vertex in $\pfi{y}$ by one, we get
       a tighter embedding of $(G,L)$, contradicting the fact that we started with a tigthest 
       embedding.  
     \end{proof}

    \begin{lemma}\label{lem:rthin}
     If $G$ is $R$-thin then each $G_i$ in a tightest induced embedding of $(G,L)$ is $R$-thin. 
    \end{lemma}
    \begin{proof}
      Towards contradiction, assume that some $G_i$ contains vertices $a$ and $b$ with the
      same neighbourhoods.  As $G_i$ is a proper interval graph, we may assume that 
      $b = a + 1$. 
      By Claim \ref{cl:6-2} no non-tight cover projects onto $a \prec a+1$ so there is some 
      $x \in \cG$ with $x_i = a$ such that $x$ and $x + e_i$ are both in $G$.
      But these have the same  neighbourhoods in $\cG$, and so as $G$ is an induced 
      subgraph, they have the same neighbourhood in $G$, contradicting the fact that 
      $G$ is $R$-thin.
    \end{proof}

   \subsection{Removing dispensible edges to get the subgraph $S$}

   The following definition can also be found in \cite{HIK}.

    \begin{definition}\label{def:disp}
    An edge $x \sim y$ of $G$ {\em dispensable} if it satisfies the following conditions. 
    \begin{enumerate}
     \item $\exists z$ such that $N(x) \subsetneq N(z) \subsetneq N(y)$, or
     \item $\exists z$ such that $N(y) \subsetneq N(z) \subsetneq N(x)$, or
     \item $\exists z$ such that $N(x) \cap N(y)  \subsetneq N(x) \cap N(z)$ and $N(x) \cap N(y) \subsetneq N(y) \cap N(z)$.
    \end{enumerate}   
    Observe that when $G$ is $R$-thin, we can replace the $\subsetneq$ in the first two conditions
    with $\subset$; they are equivalent.
  \end{definition}

    In the proof of the following lemma we will assume an embedding of $(G,L)$ into some
    product $\cG = \prod G_i$ of proper inteval graphs $G_i$. We will ust the following  
    notation. For a vertex  $v_i$ of $G_i$ we let $v_i^+ = \max\{ N_{G_i}(v_i) \}$ be the
    greatest
    neighbour of $v_i$ in $G_i$ and $v_i^- = \min\{ N_{G_i}(v_i) \}$ be its least neighbour.
    As $G_i$ is proper interval $v_i \leq u_i$ implies that $v_i^+ \leq u_i^+$ and     
    $v_i^- \leq u_i^-$. As $G_i$ is $R$-thin, strict inequality $v_i < u_i$ implies 
    strict inequality in at least one of $v_i^+ \leq u_i^+$ and     
    $v_i^- \leq u_i^-$.

    \begin{lemma}\label{lem:HHasse}
       Let $(G,L)$ be a $\DL$-pair, $G$ be $R$-thin, and $S$ be the graph we get from $G$
       by removing all dispensable edges.
       Then 
       \begin{enumerate}
        \item[(a)] Every edge of $S$ is between comparable vertices of $L$, and
        \item[(b)] $S$ contains the Hasse graph $H(L)$ of $L$.
        \end{enumerate}
     \end{lemma}   
     \begin{proof}
       First we prove part (a),  by showing that any edge
       $xy$ between incomparable vertices $x$ and $y$ is dispensable.
       Indeed, as $x$ and $y$ are incomparable, we have that $x \wedge y$ and $x \vee y$
        are distinct and different from $x$ and $y$.
        Further as $\wedge$ and $\vee$ are polymorphisms, any common neighbour of $x$
        and $y$ is a neighbour of both of $x \wedge y$ and $x \vee y$, so 
       $N(x) \cap N(y) \subseteq N(x \wedge y), N(x \vee y)$.  By $R$-thinness, 
       $N(x \wedge y)$ and $N(x \vee y)$ are distinct, so one of them properly contains
       $ N(x) \cap N(y)$. Thus $xy$ is dispensable. 

       Now, part (b) is harder. Let $x \prec y$ be a cover of $L$; we show that it is not
       dispensable. Assume some tightest induced embedding of $(G,L)$ into a product 
       $\cG$ of proper interval graphs.  We have two cases.\\
      
    \n   {\bf Case: $x \prec y$ is a not tight cover.} For any $z$ in $L \setminus \{x,y\}$, we may
       assume by Claim \ref{cl:6-1} that $x \prec y < z$. 
       We show that item i. of Definition \ref{def:disp} cannot hold. Any neighbour of $y$ in 
       $\pfi{y}$ is a neighbour of $z$ by \eqref{id:vee} (of Section \ref{sec:def}) and so
       any $w$ in $N(y) \setminus N(z)$ must be in $\pid{x}$. But then by \eqref{id:minmax} it
       is adjacent to $x$, contradicting that $N(x) \subset N(z)$.
       Items ii. and iii.  cannot hold, as by \eqref{id:minmax} any common neighbour of $x$
       and $z$ is also a neighbour of $y$. \\

   \n   {\bf Case: $x \prec y$ is a tight cover.} 
      By the $R$-thinness of $G$, we may assume without loss of generality that there is some
      $v \in N(y) \setminus N(x)$.  So immediately, condition (ii) of Definition \ref{def:disp}
      does not hold. 
      We may assume, by permuting indices of the interval graphs $G_i$, and possibly reversing 
      the ordering on the first one, that $y_1 = x_1 + 1$  
        
        Assume that (i) holds, that is, that there is some $z$ with 
        $N(x) \subsetneq N(z) \subsetneq N(y)$. As $z$ has some neighbour in $N(y) \setminus N(x)$
        we get that $x_1 < y_1 \leq z_1$, so in particular $y_1^+ < z_1^+$.  As $x \sim y \sim z$
        we have $z_1^- \leq y_1 \leq x_1^+$.  
        There is some vertex $w$ adjacent to $y$ but neither $z$ nor $x$. As it is in 
        $N(y) \setminus N(x)$ we have $x_1^+ < w_1 < y_1^+$. Putting these together we have
         \[  z_1^- \leq y_1 \leq x_1^+ < w_1 \leq y_1 < z_1^+\]
        and so $w \not\sim z$ means we may assume that $w_2 \not\sim z_2$, and
        so that $z_2 < y_2 \leq z_2^+ < w_2$.

        Now let $w'$ be the vertex in $\cG$ we get from $x$ by replacing $x_1$ with $x_1^+$
        and $x_2$ with $z_2^+ + 1$.  So $w' \not\sim z$ in $\cG$, while  $w \sim x$ and 
        $x_2 = y_2 < w'_2 \leq w_2$ implies that $x_2 \sim w'_2$, so $w' \sim x$ in $\cG$.
        As $N(x) \subset N(z)$ in $G$, $w'$ cannot be in $G$, so is in some vertex interval
        $\nVBs$ removed from $\cG$ to get $G$.  

        As $w$ is not in $\nVBs$ we have that $j = 1$ and $x_1^+ \leq \beta < w_1$.
        Also,  as $x$ is not in $\nVBs$, we have that $i = 1$ or $2$. But a tightest embedding
        cannot have a vertex interval of the form $\nVB{\alpha}i{\beta}i$ removed, so
        $i = 2$ and $x_2 < \alpha \leq w'_2 = z_2^+ + 1$.  So $z_2 < x_2 < \alpha \leq z_2^+ + 1$.

 
        Now we claim that the vertex $x'$ which we get from $x$ by reducing $x_2$ to 
        $z_2$ has the same neighbourhood as $x$ in $G$, a contradiction.
        Indeed $N(x')$ contains $N(x) \cap N(z)$, so contains $N(x)$,  
        and all vertices $v$ of $\cG$ that are adjacent to $x$ but not $x'$ have 
        $ \alpha \leq z_2^+ + 1 \leq v_2$ and $v_1 \leq x_1^+ \leq \beta$ so are in
        $\nVB{\alpha}2{\beta}1$ which has been removed. 
        Thus we have our contradiction, so (i) cannot hold.

        The argument that (ii) cannot hold is essentially the same. 
        Finally, assume that (iii) holds.  Clearly this implies that both 
        $N(x) \setminus N(y)$ and $N(y) \setminus N(x)$ are non-empty, so
        $x_1^+ < y_1^+$ and $x_1^- < y_1^-$.  Moreover, $z$ has a neighbour 
        $a \in N(Y) \setminus N(X)$, so having $a_1 > x_1^+$, and similarly another 
        neighbour $b$ having $b_1 < y_1^-$. But then there is no viable value for $z_1$.  

        This completes the proof of (b) and so of the lemma. 
      \end{proof}

    Compare Lemma \ref{lem:HHasse} to similar statements in \cite[Chap 8]{HIK},
    where they show that $S$ is closely related to what they call the 
    {\em Cartesian skeleton} of a product graph $G$. 
    Our proof is complicated by the fact that $G$ is not a product, but a subgraph of a product. 

    \subsection{Orienting edges of  $S$}

     \begin{algorithm}\label{alg:2}
      Given a graph $G$ and subgraph $S$  with designated vertices $\zero$ and $\unit$,
      let the sets $N_j$ and the graphs $D_j$ for all $j = 0, \dots, \dist(\unit,\zero)$
      be defined as follows. 
      $N_j = \{ v \in G \mid \dist(\unit,v) = j\}$, and $D_j$ is be the subgraph of $S$
      induced by $\cup_{\alpha = 0}^j N_j$. 
      Let $\vec{S}$ be the partial orientation of $S$ we get as follows.  
      For $j = 1, \dots,  \dist(\unit,\zero)$ do the following. For an edge $uv$ of 
      $S$, let $u \to v$ if 
      \begin{enumerate}
            \item  $u \in N_{j-1}$ and $v \in N_j$, or
            \item  $u, v \in N_j$ and any one of the following holds 
      \begin{enumerate}
           \item $N(u) - N(v)$ has a vertex $u' \in D_{j-1}$ such that 
                 $u' > v'$ for all $v' \in N(v) \cap D_{j-1}$. 
                 (We consider $u' > v'$ if there is a directed path in $D_{j-1}$
                  from $u'$ to $v'$.)
           \item $N(v) - N(u)$ has a vertex $v' \in D_{j-1}$ such that
                 $u' > v'$ for all $u' \in N(u) \cap D_{j-1}$. 
           \item $N(v) - N(u)$ has a vertex in $N_j \cup N_{j+1}$ but not in $N_{j-1}$.   
         \end{enumerate}           
      \end{enumerate}        
     If every edge of $S$ is oriented in $\vec{S}$,  and the transitive closure of $\vec{S}$ is
     a lattice $L$, then return $L$, otherwise, return 'NO'.
    \end{algorithm}
   
      This algorithm is clearly polynomial in $n$. 

     \begin{lemma}\label{lem:orient}
       Let $(G,L)$ be a compatible pair; $G$ be $R$-thin; and let $S$ be a subgraph of $G$, 
       containing the Hasse graph $H(L)$ of $L$, such that every edge of $S$ is between
       vertices that are comparible in $L$. 
       Algorithm \ref{alg:2}, applied to $S$, $\unit_L$, and $\zero_L$, returns $L$.  
     \end{lemma}
      \begin{proof}
         It is enough to show that for any (non-loop) edge $uv$ of $S$ with $u > v$,
         the above algorithm {\em properly orients $uv$}; i.e.,  sets $u \to v$ and at
         the same time does not set $v \to u$.

         Observe that by construction every edge of $S$ is either in $D_j$ for some $j$
         or is between $D_{j-1}$ and $D_{j}$ for some $j$.
         We will prove by induction on $j$
         that the $j^{th}$ step of the algorithm proper orients such edges, yielding 
         a proper orientation of all the edges of $D_j$. Before we do this though,
         we first prove that it will never improperly orient an edge.

         \begin{claim}
           Let $u > v$ then the algorithm will not set $v \to u$.
         \end{claim}
         \begin{proof} \claimproof
           We must check that none of the conditions of the algorithm are satisfied when the roles 
           of $u$ and $v$ are reversed

           To see that item (i) is not satisfied observe that if not both of $u$ and $v$ are in $N_j$,
           then clearly it is $u$ that is closer to $1$. Indeed,  if 
           $v = x_\l \sim x_{\l- 1} \sim \dots \sim x_1 = 1$ is a path in $G$, then so is 
           $u = u \vee x_\l \sim u \vee x_{\l - 1} \sim \dots \sim u \vee x_1 = 1$.
           So $u \in N_{j-1}$ and $v \in N_j$.  ( In fact this shows that the algorithm properly sets
           $u \to v$ in the case that $u$ and $v$ are not both in $N_j$. 

           To see that items (iia) and (iib) are not satisfied, it is enough to observe that if 
           $u' \sim u$ and $v' \sim v$ and $v' \geq u'$ then $u' \sim v$ and $u \sim v'$. 
           But this is clear, as the premises imply that 
              \[ u' = u' \wedge v' \sim u \wedge v = v \]
           and 
             \[ v' = u' \vee v' \sim u \vee v  = u. \]

           To see that item (iic) is not satisfied, assume that there is some 
           $w \in N(u) - N(v)$.  As $N(u)$ is conservative, (recall the definition 
          of conservative sets preceding Lemma \ref{lem:connected}) it induces a sublattice
          of $L$, so has a maximum element $u'$.  This element must also be in $N(u) - N(v)$;
          as if we had $w' \sim v$, then 
           \[ w = w \vee u' \sim u \vee v = v, \]
          contradicting the fact that $w \not\in N(v)$. 
          We now show that $u'$ is in $N_{j-1}$, so item (iic) is not satisfied. 
          Indeed, some neighbour $x$ of $u$ must be in $N_{j-1}$, as $u \in N_j$. 
          Let $x = x_{i-1} \sim x_{i-2} \sim \dots \sim x_0 = \unit$ be a length $i-1$
          walk from  $x$ to $\unit$.  Then taking the join of each element in the 
          walk with $u'$ we get a walk 
              $u' = u' \vee x_{i-1} \sim u' \vee x_{i-2} \dots \sim u' \vee \unit = \unit$
         from $u'$ to $\unit$. This shows that $u'$ is in $N_i$ for some $i \leq j-1$, but
         being a neighbour of $u$, it must be in $N_{j-1}$.  
         \end{proof}

         Now we have just to verify that for $u > v$ the algorithm sets $u \to v$. 
                   
         For the case $j = 1$ let $uv$ be an edge of $S$ in $D_1$ with $u > v$. 
         Item (i) holds if and only if $u = \unit$, and in this case gives $u \to v$, as needed.
         Assume therefore that $u,v \in N_1$. 
         As all vertices in $N_1$ are adjacent to $\unit$,
         items (iia) and (iib) are vacuous, so we must show that (iic) holds.
         To see this, observe that as $u \geq v$, we have that 
         $u_i \geq v_i$ for all $i \in [d]$.
         As $u_i^+ = 1 = v_i^+$ for all $i$, we have by $R$-thinness that
         $N_G(1) \subsetneq N_G(u) \subsetneq N_G(v)$.
         The vertex in $N_G(v) \setminus N_G(u)$ is thus
         in $N_2$ as needed.
  
         Now assume that all edges of $D_{j-1}$ are properly oriented. 
         We show that the $j^{th}$ round of the algorithm properly orients the 
         heretofore unoriented edges of $D_j$.

         Let $uv$ be an edge of $D_j \setminus D_{j-1}$  with $u \geq v$. 
         If not both of $u$ and $v$ are in $N_j$, then (as we showed in the claim) 
         $u \to v$ is properly ordered by step (i) of the algorithm. 

         So we may assume that both of $u$ and $v$ are in $N_j$.
         As $u > v$ we have that for all $i$, $u_i \geq v_i$. 
         By $R$-thinness there is a vertex $w$ in either 
         $N(u) \setminus N(v)$ or in $N(v) \setminus N(u)$.
         
         We show now that in the first case, (iia) is satisfied, and then that in the 
         second case, (iib) or (iic) are satisfied.
         
         \begin{claim}
           If $w \in N(u) \setminus N(v)$, then (iia) is satisfied.
         \end{claim}
         \begin{proof}\claimproof
            Let $w \in N(u) \setminus N(v)$.
            As we showed in the proof that item (iic) is not satisfied in the previous claim, 
            we have that the maximum neighbour $u'$ of $u$ is in  
            $D_{i-1} \cap (N(u) \setminus N(v))$.  
       
            To see that (iia) is satisfied,  we must show that $u' \geq v'$ for 
            any neighbour $v'$ of $v$ in $D_{i-1}$. Indeed, 
            $v' \sim v$ and $u' \sim u$  give 
            $v' \vee u' \sim v \vee u = u$. 
            As $u'$ is the maximal neighbour of $u$ this gives us that 
            $u' \geq v' \vee u'$. This implies however that 
            $u' = v' \vee u'$, and so $u' \geq  v'$, as needed. 

         \end{proof}

        \begin{claim}
          If $w \in N(v) \setminus N(u)$, then (iib) or (iic) are satisfied.
        \end{claim} 
        \begin{proof} \claimproof
          We assume that (iic) does not hold, and then show that (iib) must.
          
          Indeed, if (iic) does not hold, then  
          $w \in D_{j-1} \cap (N(v) - N(u))$. As $D_{j-1} \cap N(v)$ is a conservative
          set it induces a sublattice, so has a minimum element $v'$.
          But then for any neighbour $u'$ of $u$ in $D_{j-1}$ we have
          from $v' \sim v$ and $u' \sim u$, that $v' \wedge u' \sim v \wedge u = v$. 
          As $D_{j-1}$ is conservative, $v' \wedge u'$ is in $D_{j-1}$ so is in $D_{j-1} \cap N(v)$.
          Thus $v' \wedge u' \geq v'$ which implies that $u' \geq v'$. As $v' \not\in N(u)$ 
          we have that $u' > v'$, as needed.
        \end{proof}
        
       This completes the proof of the lemma. 

      \end{proof}

      Now we are ready to prove Theorem \ref{thm:poly}.
  
     \begin{proof}[Proof of Theorem \ref{thm:poly}]
       Let $G$ be an $R$-thin graph. It is shown in \cite{HIK} that the subgraph $S$ we get by 
       removing dispensible edges can be found in polynomial time. 
       For every choice of $\zero$ and $\unit$ in $G$ apply Algorithm \ref{alg:2} to 
       $G,S$ and $\zero$ and $\unit$. As this is at most $n^2$ applications of a polynomial
       time algorithm, it is also polynomial. If for any choice of $\zero$ and $\unit$ a lattice 
       $L$ is returned then by Lemmas \ref{lem:HHasse} and \ref{lem:orient},  then $G$ is in $\DL$ and
       $L$ is a compatible distributive lattice. 
       If 'No' is returned for every choice, then by Lemmas \ref{lem:HHasse} and 
       \ref{lem:orient}, $G$ is not in $\DL$. 
     \end{proof}

      Here is an unexpected consequence of our algorithm.

     \begin{corollary}\label{cor:uniqueType2}
        If $G$ is $R$-thin, then for a given choice of minimum and maximum
        vertices $\zero$ and $\unit$, there is at most one distributive lattice $L$ with
        minimum element $\zero$ and maximum element $\unit$ that is compatible 
        with $G$. 
     \end{corollary}

     With this we can get the following.

  \begin{figure} \setlength{\unitlength}{280bp}
    \begin{center}
    \begin{picture}(1,0.43333335)(0,0)%
    \put(0,0){\includegraphics[width=\unitlength]{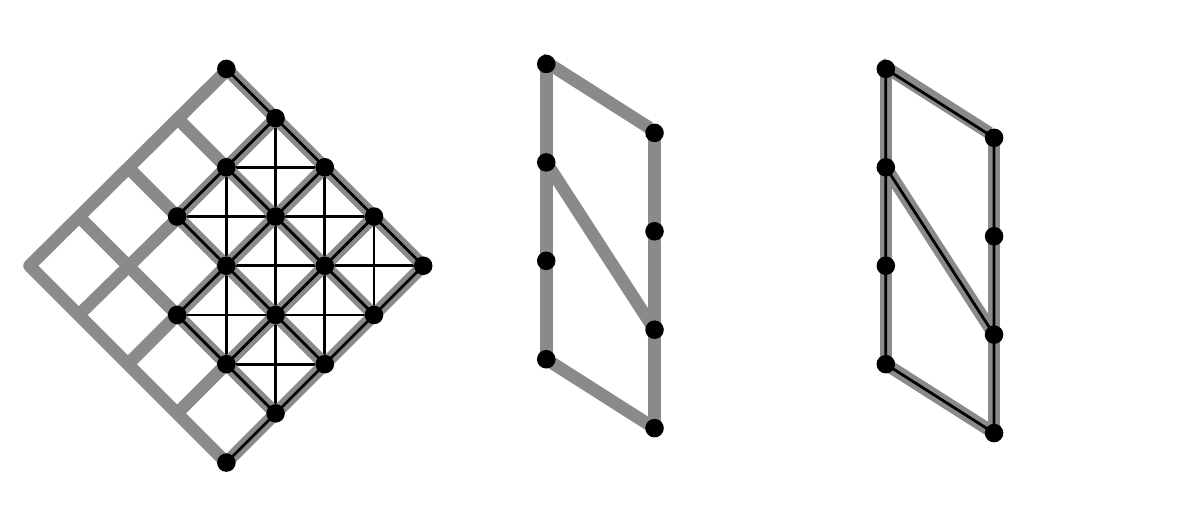}}%
    \put(0.47,0.017){\color[rgb]{0,0,0}\makebox(0,0)[lb]{\smash{$J_L$ }}}%
    \put(0.72500005,0.017){\color[rgb]{0,0,0}\makebox(0,0)[lb]{\smash{$\redcomp{A}$}}}%
    \put(0.21,0.017){\color[rgb]{0,0,0}\makebox(0,0)[lb]{\smash{$(G,L)$}}}%
    \put(0.20833334,0.37500001){\color[rgb]{0,0,0}\makebox(0,0)[lb]{\smash{$\unit$}}}%
    \put(0.16,0.02500002){\color[rgb]{0,0,0}\makebox(0,0)[lb]{\smash{$\zero$}}}%
    \put(0.40,0.12083341){\color[rgb]{0,0,0}\makebox(0,0)[lb]{\smash{$\ce11$}}}%
    \put(0.40,0.21214294){\color[rgb]{0,0,0}\makebox(0,0)[lb]{\smash{$\ce21$}}}%
    \put(0.40,0.29865747){\color[rgb]{0,0,0}\makebox(0,0)[lb]{\smash{$\ce31$}}}%
    \put(0.40,0.38097229){\color[rgb]{0,0,0}\makebox(0,0)[lb]{\smash{$\ce41$}}}%
    \put(0.57246034,0.06148159){\color[rgb]{0,0,0}\makebox(0,0)[lb]{\smash{$\ce12$}}}%
    \put(0.58084656,0.15098546){\color[rgb]{0,0,0}\makebox(0,0)[lb]{\smash{$\ce22$}}}%
    \put(0.57916668,0.23749999){\color[rgb]{0,0,0}\makebox(0,0)[lb]{\smash{$\ce32$}}}%
    \put(0.5774868,0.31981481){\color[rgb]{0,0,0}\makebox(0,0)[lb]{\smash{$\ce42$}}}%
    \end{picture}%
  \caption{Compatible pair $(G,L)$, poset $J_L$, and the graph $\redcomp{A}$}
  \label{fig:3not2}
  \end{center}
  \end{figure}

    \begin{proposition}\label{prop:3not2}
      There are distributive lattice graphs that  are not tight.  
    \end{proposition} 
    \begin{proof}
      Let $(G,L)$ be the $\DL$-pair shown with a tight but non-induced embedding
      in Figure \ref{fig:3not2}. 
      By Example \ref{ex:degOne}, the shown $\zero$ and $\unit$ are the only 
      possible $\zero$ and $\unit$ for lattice $L$ compatible with $G$. 
      As $G$ is $R$-thin,  we have
      by Corollary \label{cor:uniqueType2} that $L$ is the only distributive 
      lattice (upto isomorphism) compatible with $G$.  
     
     The poset $J_L$ and subdigraph $\redcomp{A}$ are also shown. 
     As $\redcomp{A}$ has a vertex with up-degree two, we have, following
     remarks in Subsection \ref{sub:nottight}, that there is no tight induced
     embedding of $(G,L)$. So $G$ has no compatible distributive lattice with which it has
     a tight induced embedding. 
    \end{proof}



  \subsection{Non $R$-thin graphs}\label{sub:NRThin}

   For a reflexive graph $G$ we define a relation $R$ on the vertex set by letting $uRv$ if 
   $u$ and $v$ have the same neighbourhood. Clearly this is an equivalence relations. 
   The {\em $R$-thin reduction} of a graph $G$ is the graph $G^R$ whose vertices are the sets  
   $R$ and in which two sets are adjacent if there are any (and so all) edges between their
   member vertices.  

   The following shows our algorithm can be useful in showing that a non $R$-thin graph is not
   $\DL$.

   \begin{lemma}\label{lem:rthinred}
      If a reflexive graph $G$ is a $\DL$-graph then its $R$-thin reduction is. 
   \end{lemma}
   \begin{proof}
     Assume that $G$ is a reflexive $\DL$-graph that is not $R$-thin.
     We will find pairs of vertices that are identified in $G^R$ and show that when we 
     identify them, or reduce the number of coordinates in which they differ,  we still
     have a $\DL$ graph. The fact that $G^R$ is $\DL$ then follows by induction.  

     For some compatible $L$ assume a tightest induced embedding of $(G,L)$ into $(\cG,\cP)$. 
     Let $x$ and $y$ be vertices of $G$ with the same neighbourhood. By     
     Lemma \ref{lem:rthin} there is some $G_j$ such that $x_j$ and $y_j$ have the same 
     neighbourhood in $G_j$. We may assume that $x_j = y_j - 1$. For any vertex $v$ in 
     $G$ with $v_j \geq y_j$, reduce $v_j$ by $1$. If under this reduction, two vertices
     now have the same co-ordinates, then identity them- they had the same neighbourhood
     so are identified in $G^R$. Clearly there is an embedding of this reduced graph into
     $\cG' = \prod G'_i$ where $G'_i = G_i$ when $i \neq j$ and we get $G'_j$ from $G_j$
     by identifying $x_j$ and $y_j$.  
   \end{proof}

   We conjecture the following.   

   \begin{conjecture}\label{conj:rthin}
    For a reflexive graph $G$ there is a polynomial time algorithm to decide whether
    or not $G$ is a $\DL$-graph. 
   \end{conjecture} 

   Notice that the graph $G$ in Figure \ref{fig:LnotDL} is not $R$-thin. The vertices $d$ and $e$ have
   the same neighbourhoods. If we remove one of these vertices then the resulting lattice is 
   distributive and is still compatible with the resulting $R$-thin reduction $G^R$. Thus the 
   converse of the above lemma is, unfortunately, not true.
   
   That said, one sees  by reversing the operation in the proof of 
   Lemma \ref{lem:rthinred} that from an embedding of a $\DL$-pair, we can add a copy of
   every vertex that has the same value in some coordinate. Moreover one can argue that 
   'fattening' the lattice in a new dimension can be replicated in the existing dimensions.
    So resolving the conjecture comes down solving a general version of the following game, 
   described vaguely, but clear from Figure~\ref{fig:Game}.

   Given a set of numbers in a diamond tableau, decide if one can  
    \begin{itemize}
      \item divide the regions with square lines, and
      \item make two decreasing walks from the top to the bottom,
    \end{itemize}    
    so that the number of divided regions in each of the original regions between the walks
    equals the number proscribed in the tableau. 
    With some students \cite{HPS}, we show that this game has a polynomial time solution for tableaux of
    two dimentions.

  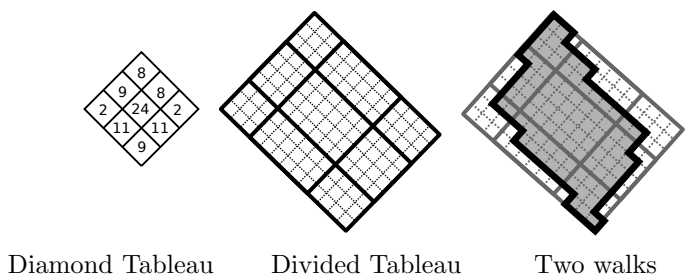
\begin{figure} 
    \begin{center}
    \begin{picture}(8,6)(0,0)%
    \put(0,.5){\includegraphics[width=8cm]{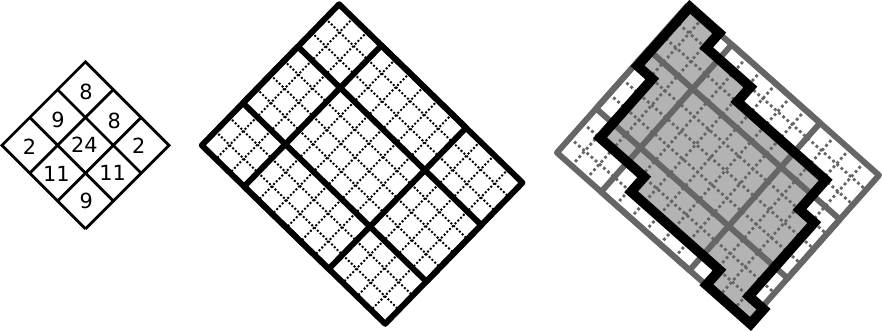}}%
    \put(-1,0){Diamond Tableau}
    \put(2.5,0){Divided Tableau}
    \put(6,0){Two walks}
    \end{picture}
  \caption{The Game of Conjecture \ref{conj:rthin}}
  \label{fig:Game}
  \end{center}
  \end{figure}

  \section{A Question} 

  A partial characterisation of lattice graphs can be extracted from known literature. 
  Indeed, it follows from \cite{HeThes} and \cite{JPM} (see also \cite{FHLLST}) that 
  retracts of products of reflexive paths are exactly the reflexive graphs that admit
  majority, or
  $3$-NU polymorphisms, that is, polymorphisms $f:V(G)^3 \to V(G)$ satisfying  
  \begin{quote}
    $f(x,y,z) = c$ if at least two of $x,y$ and $z$ are $c$. 
  \end{quote}  
   If a reflexive graph has a compatible lattice, then it also admits the
   following majority operation (seen, for example, in \cite{Band}) 
          \[ f(x,y,z) = (x \wedge y) \vee (y \wedge z) \vee (x \wedge z).\]
   Thus all lattice graphs are retractions of products of paths. 
   It would also be nice to see how our characterisations can be use to show this for $\DL$-graphs:
   that every $\DL$-graph is a retract of products of paths. 
   In general, removing a vertex interval $\nVBs$ is not a retraction. 



  \providecommand{\bysame}{\leavevmode\hboxto3em{\hrulefill}\thinspace}


\begin{thebibliography}{15}

  \bibitem{Band}
  H. Bandelt,
  {\em Graphs with edge-preserving majority functions.}
  Discrete Math. 103 (1992) pp 1--5

   \bibitem{Bi}
   G. Birkhoff 
   {\em Rings of sets}
   Duke Math. Jour.  3 (3) (1937) pp 443--454.
   (doi:10.1215/S0012-7094-37-00334-X).
  


    \bibitem{BJK}
     A. Bulatov, P. Jeavons, A. Krokhin,
     {\em Classifying the complexity of constraints using finite algebras.}
     SIAM J. Comput.  34  (2005),  no. 3, 720--7


    \bibitem{CDK}
    C. Caravalho, V. Dalmau, A. Krokhin.
    {\em Caterpillar duality for constraint satisfaction problems}
     Proceed. LICS '08 (2008) pp. 307-316 

    \bibitem{CKNOS}
    D. Corneil , H. Kim , S. Natarajan , S. Olariu , A. Sprague.
    {\em Simple linear time recognition of unit interval graphs}
    Information Processing Letters,  55 (2)  (1995), pp. 99--104.


    \bibitem{Di50}
    R. Dilworth.
    {\em A decomposition theorem for partially ordered sets}
    Annals of Math. 51 (1950) pp. 161-166.


   \bibitem{DI70}
    W. Dorfler, W. Imrich,
    {\em \"Uber das starke Produkt von endlichen Graphen.}
    \"Osterreich. Akad. Wiss. Math.-Nature Kl.S.-B. II 18, pp.247-262. 

 




\bibitem{FV98} T.~Feder and M.Y. Vardi.
  {\em The computational structure of monotone monadic {SNP} and constraint satisfaction: A study through {D}atalog and group theory.}
  SIAM Journal on Computing, 28, (1998), pp. 57--104.



    \bibitem{FS86}
    J. Feigenbaum, A.A.Sch\"affer,
    {\em Finding the prime factors of strong direct product graphs in polynomial time}

    \bibitem{FHLLST}
     T. Feder,  P. Hell,  B. Larose,  C. Loten,  M. Siggers, C. Tardif,
     {\em Graphs admitting k-NU operations. Part 1: The Reflexive Case.}
     SIAM J. Discrete Math. 27(4), (2013) pp. 1639-2166. 


   

     \bibitem{Ga}
   F. Gardi,
  {\em  A note on the Roberts characterization of
proper and unit interval graphs. }
   Discrete Math. 307, (2007), pp. 2906--2908.

  

     \bibitem{HIK}
     R. Hammack, W. Imrich, S. Klac\v zer,
     {\em Handbook of product graphs.}
     CRC Press (2nd Edition 2011).

  
    \bibitem{HeThes}
    P. Hell 
    {\em R\'etractions de graphes}
    Ph.D. Thesis, Universit\'e de Mont\'eal (1972).

    \bibitem{HS}
    P. Hell, M. Siggers,
    {\em Semilattice Polymorphisms and Chordal Graphs}
    Euro. Journ. of Comb. 36 (2014), pp 694--706. 

   \bibitem{HPS}
    D.Y. Hong, S.J. Pi, M. Siggers,
    {\em A solution to the two-dimensional lattice blow-up game}
    Manuscript.

   \bibitem{JPM} E. Jawhari, M. Pouzet, D. Misane, 
    {\em Retracts: graphs and ordered sets from the metric point of view}
   Combinatorics and ordered sets (Arcata, Calif., 1985), 175--226, Contemp. Math., {\bf 57}, Amer. Math.
   Soc., Providence, RI, 1986.

  


  
    \bibitem{Ri74}
    I. Rival,
    {\em Maximal sublattices of finite distributive lattices. II.}
    Proc. Amer. Math. Soc.  44 (1974) pp. 263-268.

    \bibitem{Si15}
    M. Siggers,
    {\em On the representation of finite distributive lattices}
    Submitted (Nov. 2014) {\tt arXiv:1412.0011 [math.CO]}.

   
 \end{thebibliography}
\end{document}